\documentclass[oneside,english]{amsart}
\usepackage{graphicx,amsmath,amsthm,verbatim,amssymb,lineno,titletoc}
\usepackage{geometry}
\usepackage{bbm}
\usepackage{bm}
\usepackage{hyperref}
\hypersetup{colorlinks}
\usepackage[numbers]{natbib}
\usepackage{array}
\usepackage{setspace}
\pdfoutput=1

\newtheorem{theorem}{Theorem}[section]
\newtheorem{prop}[theorem]{Proposition}

\newtheorem{lemma}[theorem]{Lemma}

\theoremstyle{remark}
\newtheorem{remark}[theorem]{Remark}

\numberwithin{counter}{section}

\theoremstyle{definition}
\newtheorem{example}[theorem]{Example}

\def\N{\mathbb{N}}
\def\Q{\mathbb{Q}}
\def\one{\mathbf{1}}
\def\zero{\mathbf{0}}
\newcommand{\T}{\mathsf{T}}
\newcommand{\e}{\varepsilon}
\newcommand{\st}{:\,}

\allowdisplaybreaks

\begin{document}

\title{Positional Voting and Doubly Stochastic Matrices}

\author{Jacqueline Anderson}
\address{Jacqueline Anderson, Department of Mathematics, Bridgewater State University, Bridgewater, MA 02325.}
\email{\texttt{Jacqueline.Anderson@bridgew.edu}}
\urladdr{\url{http://webhost.bridgew.edu/j21anderson/}}

\author{Brian Camara}
\address{Brian Camara, Department of Mathematics, University of Rhode Island, Kingston, RI 02881.}
\email{\texttt{briancamara@my.uri.edu }}

\author{John Pike}
\address{John Pike, Department of Mathematics,  Bridgewater State University, Bridgewater, MA 02325.}
\email{\texttt{john.pike@bridgew.edu}}
\urladdr{\url{http://webhost.bridgew.edu/j1pike/}}

\date{\today}


\begin{abstract}
We provide elementary proofs of several results concerning the possible outcomes arising from a fixed profile 
within the class of positional voting systems. Our arguments enable a simple and explicit construction of paradoxical profiles, 
and we also demonstrate how to choose weights that realize desirable results from a given profile. The analysis ultimately boils 
down to thinking about positional voting systems in terms of doubly stochastic matrices.
\end{abstract}

\maketitle

\section{Introduction}
\label{Introduction}

Suppose that $n$ candidates are running for a single office. There are many different social choice procedures one 
can use to select a winner. In this article, we study a particular class called \emph{positional voting systems}. A positional 
voting system is an electoral method in which each voter submits a ranked list of the candidates. Points are then assigned 
according to a fixed \emph{weighting vector} $\mathbf{w}$ that gives $w_i$ points to a candidate every time they appear in 
position $i$ on a ballot, and candidates are ranked according to the total number of points received. For example, plurality is a 
positional voting system with weighting vector 
$\mathbf{w} = [\,\setlength\arraycolsep{3pt}\begin{matrix} 1 & 0 & 0 & \cdots & 0\end{matrix}\,]^\T$. One point is 
assigned to each voter's top choice, and the candidate with the most points wins. The Borda count is another common 
positional voting system in which the weighting vector is given by 
$\mathbf{w} =[\,\setlength\arraycolsep{3pt}\begin{matrix} n-1 & n-2 & \cdots & 1 & 0\end{matrix}\,]^\T$. 
Other examples include the systems used in the Eurovision Song Contest, parliamentary elections in Nauru, and the 
AP College Football Poll \cite{BesRob,FraGro,HodKlim}.

By tallying points in this manner, a positional voting system outputs not just a winner, but a complete (though not necessarily strict) 
ranking of all candidates, called the \emph{societal ranking}. The societal ranking produced by a positional voting system depends 
not only on the set of ballots (called the \emph{profile}) but also on the choice of weighting vector. With the freedom to choose different 
weighting vectors, one can achieve many different outcomes from the same profile. An immediate question is 
``Given a profile, how many different societal rankings are possible?''

In \cite{Saari}, Donald Saari famously showed that for any profile on $n$ alternatives, there are at most $n!-(n-1)!$ 
possible strict societal rankings depending on the choice of weighting vector. Moreover, this bound is sharp in that there 
exist profiles for which exactly this many different strict rankings are possible. In \cite{DEMO}, Daugherty, Eustis, Minton, and 
Orrison provided a new approach to analyzing positional voting systems and extended some of Saari's results to 
cardinal (as opposed to ordinal) rankings, as well as to partial rankings; see also \cite{CrisOrr}. Our article serves to 
complement these works by providing alternative derivations that afford more explicit constructions, new perspectives, 
and arguments which some may find more accessible. 
 
After establishing some conceptual foundations, we proceed by proving the main result in \cite{DEMO} using facts about 
doubly stochastic matrices (Theorem~\ref{thm:main}). With a little more linear algebra, we are then able to recover 
Saari's findings (Theorems \ref{number rankings lower bound} and \ref{number rankings upper bound}). An advantage of 
our methodology is that it gives a concrete means of constructing ``paradoxical profiles.'' It also enables us to provide a 
simple geometric characterization of the possible outcomes resulting from a given profile (Theorem~\ref{convex hull}).

In addition, our work illustrates the utility of thinking about doubly stochastic matrices and the braid arrangement in 
problems related to social choice procedures. Doubly stochastic matrices arise very naturally in our analysis, and the 
braid arrangement provides a nice geometric realization of rankings. While our arguments do not depend on any deep facts about 
hyperplane arrangements, many of the objects we work with have a natural interpretation within this framework, and it provides  
a useful vernacular for thinking about such matters. As with the algebraic voting theory from \cite{DEMO}, the hope is that by 
formulating problems in terms of different mathematical constructs, new tools and perspectives become available. The 
connection with hyperplane arrangements has received some attention in previous works---for instance, Terao's proof of Arrow's 
impossibility theorem \cite{Terao}---but doubly stochastic matrices seem to have been given much less consideration in the context 
of voting.

Finally, our approach serves to partially bridge the perspectives from \cite{Saari} and \cite{DEMO}. Saari 
provides a geometric realization of positional voting procedures in terms of certain high-dimensional simplices; 
see also \cite{Saari95}. We use some similar ideas, but cast them in the language of the braid arrangement. 
On the other hand, Daugherty et al. describe positional voting using Young tableaux, which enables them to harness 
results from the representation theory of the symmetric group. Our analysis employs ideas from linear algebra to 
capture the essence of these arguments.


\section{Notation and Terminology}
\label{Notation}

Throughout this article, $n \geq 3$ is a fixed integer representing the number of candidates. The number of voters is $N$, 
which we only assume to be rational (though Proposition~\ref{convenient p} shows that we may take $N\in\N$ if we are 
just interested in ordinal rankings). While $n$ is arbitrary and given in advance, $N$ may have some implicit constraints 
depending on the context. We work exclusively over the field $\Q$ of rational numbers, and we write $\N_{0}$ for the 
set of nonnegative integers. Vectors are always written in boldface with the components of a generic $n$-dimensional 
vector $\mathbf{v}$ denoted by $v_1, v_2, \dots, v_n$. We write $\one$ for the vector of all ones and $J$ for the matrix 
of all ones, where the dimensions are clear from context. Finally, we use the notation $1\{\cdot\}$ to represent the indicator 
function, so that, for example, $1\{i<j\}$ equals $1$ if $i<j$ and $0$ otherwise.

Let $\sigma$ be a permutation in $S_{n}$ and define the following subsets of $\Q^n$:
\begin{gather*}
V_{0}=\big\{\mathbf{x} \in\Q^{n}\st x_{1}+\cdots+x_{n}=0\big\},\\
C_{\sigma}=\big\{\mathbf{x} \in\Q^{n}\st x_{\sigma(1)}>x_{\sigma(2)}>\cdots>x_{\sigma(n)}\big\},\\
W=C_{id}\cap V_{0}=\big\{\mathbf{x} \in\Q^{n}\st x_{1}>x_{2}>\cdots>x_{n},\;x_{1}+\cdots+x_{n}=0\big\}.
\end{gather*}
The set $W$ will be of particular importance for us and can be thought of as the set of strict weighting vectors. Clearly such vectors 
should have decreasing entries, and the sum-zero normalization is essentially because adding a multiple of $\one$ to a weighting vector 
does not affect the ranking of candidates. (Further elaboration is given at the end of this section.) If we wish to allow the same point value 
to be assigned to multiple candidates, we can consider the closure $\overline{W}$. 

Now label the permutations in $S_{n}$ lexicographically according to one-line notation, and let $R_{\ell}=R_{\sigma_{\ell}}$ be the 
$n\times n$ permutation matrix corresponding to $\sigma_{\ell}$, defined by $R_{\ell}(i,j)=1\{\sigma_{\ell}(j)=i\}$. 
Given a weighting vector $\mathbf{w}\in W$, define the $n\times n!$ matrix 
$T_{\mathbf{w}} =
\big[\,\setlength\arraycolsep{3pt}\begin{matrix}
\sigma_{1}\mathbf{w} & \sigma_{2}\mathbf{w} & \cdots & \sigma_{n!}\mathbf{w}\end{matrix}\,\big]$ having 
$\ell^{\text{th}}$ column 
\[
\sigma_{\ell}\mathbf{w}:=R_{\ell}\mathbf{w}
=\big[\,\setlength\arraycolsep{3pt}\begin{matrix}
w_{\sigma_{\ell}^{-1}(1)} & w_{\sigma_{\ell}^{-1}(2)} & \cdots & w_{\sigma_{\ell}^{-1}(n)}
\end{matrix}\,\big]^{\T}.
\]
For a given profile $\mathbf{p} \in\Q^{n!}$, the \emph{results vector} for the positional voting procedure associated with 
$\mathbf{w}$ is given by 
\begin{equation}
\label{resultsvector}
\mathbf{r} =T_{\mathbf{w}} \mathbf{p} =p_{1}R_{1}\mathbf{w} +p_{2}R_{2}\mathbf{w} +\cdots+p_{n!}R_{n!}\mathbf{w} 
=Q_{\mathbf{p}} \mathbf{w}
\end{equation}
where $Q_{\mathbf{p}}$ is a convenient shorthand for $\sum_{\ell=1}^{n!}p_{\ell}R_{\ell}$. 

Each $\sigma\in S_{n}$ corresponds to the ranking of the candidates (labeled $1$ through $n$) in which candidate 
$\sigma(k)$ is the $k^{\text{th}}$ favorite. The profile $\mathbf{p}$ encodes preferences of the electorate so that 
$p_{\ell}$ is the number of voters with preference $\sigma_{\ell}$.  The $(i,j)$-entry of $Q_{\mathbf{p}}$ is thus the number 
of voters ranking candidate $i$ in $j^{\text{th}}$ place. If each voter assigns $w_{k}$ points to their 
$k^{\text{th}}$ favorite candidate, then $r_{j}$ is the total number of points given to candidate $j$. The societal ranking 
for this election procedure is $\pi\in S_{n}$ with $\mathbf{r}\in C_{\pi}$. (We are assuming for the moment that a strict 
ranking is achieved. The possibility of ties will be addressed after Example~\ref{election}.) 

Note that we have given two different ways of computing the results vector $\mathbf{r}$ in equation~\eqref{resultsvector}. 
On one hand, we have $T_\mathbf{w}$, an $n \times n!$ matrix that encodes all the possible permutations of the weighting 
vector, which can be combined with the $n!$-dimensional profile vector $\mathbf{p}$ to yield the result. On the other hand, 
we have $Q_\mathbf{p}$, an $n \times n$ matrix that encodes the number of votes each candidate receives in each place, 
which can be combined with the $n$-dimensional weighting vector $\mathbf{w}$ to yield the result. 

\begin{example}
\label{election}
Consider an election with $4$ candidates and voter preferences described by the following table.

\begin{center}
\begin{tabular}{ |c|c| } 
 \hline
 \textnormal{Voting Preference} & \textnormal{Number of Votes} \\
 \hline
$(2, 3, 4, 1)$ & $8$  \\ 
 \hline
$(1, 3, 2, 4)$ & $5$  \\
 \hline
$(4, 3, 2, 1)$ & $10$  \\
 \hline
$(2, 3, 1, 4)$ & $8$   \\
 \hline
$(4, 1, 3, 2)$ & $7$  \\
 \hline
\end{tabular}
\end{center}

\smallskip

\noindent Then 
\begin{align*}
Q_{\mathbf{p}} & = 8\begin{bmatrix}0 & 0 & 0 & 1 \\1 & 0 & 0 & 0 \\0 & 1 & 0 & 0 \\0 & 0 & 1 & 0 \end{bmatrix} + 
			    5\begin{bmatrix}1 & 0 & 0 & 0 \\0 & 0 & 1 & 0 \\0 & 1 & 0 & 0 \\0 & 0 & 0 & 1 \end{bmatrix} + 
			    10\begin{bmatrix}0 & 0 & 0 & 1 \\0 & 0 & 1 & 0 \\0 & 1 & 0 & 0 \\1 & 0 & 0 & 0 \end{bmatrix}\\ 
				& \qquad\qquad\; +
 			    8\begin{bmatrix}0 & 0 & 1 & 0 \\1 & 0 & 0 & 0 \\0 & 1 & 0 & 0 \\0 & 0 & 0 & 1 \end{bmatrix} + 
 			    7\begin{bmatrix}0 & 1 & 0 & 0 \\0 & 0 & 0 & 1 \\0 & 0 & 1 & 0 \\1 & 0 & 0 & 0 \end{bmatrix} 
			    = \begin{bmatrix}5 & 7 & 8 & 18 \\ 16 & 0 & 15 & 7 \\ 0 & 31 & 7 & 0 \\ 17 & 0 & 8 & 13
			        \end{bmatrix}.
\end{align*}

To implement the Borda count, we use 
$\mathbf{w} =[\,\setlength\arraycolsep{3pt}\begin{matrix} 1.5 & 0.5 & -0.5 & -1.5\end{matrix}\,]^\T$, 
which is obtained from $[\,\setlength\arraycolsep{3pt}\begin{matrix} 3 & 2 & 1 & 0\end{matrix}\,]^\T$ by subtracting 
$\frac{1}{4}(3+2+1+0)\one$. This yields
\[
Q_{\mathbf{p}} \mathbf{w}
 = \begin{bmatrix}5 & 7 & 8 & 18 \\ 16 & 0 & 15 & 7 \\ 0 & 31 & 7 & 0 \\ 17 & 0 & 8 & 13  \end{bmatrix}
 \cdot \begin{bmatrix}1.5 \\0.5 \\-0.5 \\-1.5 \end{bmatrix}
 = \begin{bmatrix} -20 \\6 \\12 \\2 \end{bmatrix},
\]
so that the societal ranking is $(3,2,4,1)$.

If instead we use the plurality method, $\mathbf{w}
 = [\,\setlength\arraycolsep{3pt}\begin{matrix} 0.75 & -0.25 & -0.25 & -0.25\end{matrix}\,]^\T$, we find that 
\[
Q_{\mathbf{p}} \mathbf{w}
 =\begin{bmatrix}5 & 7 & 8 & 18 \\ 16 & 0 & 15 & 7 \\ 0 & 31 & 7 & 0 \\ 17 & 0 & 8 & 13  \end{bmatrix}
 \cdot \begin{bmatrix}0.75 \\-0.25 \\-0.25 \\-0.25\end{bmatrix}
 = \begin{bmatrix}-4.5 \\6.5 \\-9.5 \\7.5 \end{bmatrix},
\]
resulting in $(4,2,1,3)$. By changing the weights, we moved the first place candidate to last place!
\end{example}

Of course, it is also possible that $\mathbf{r} \not\in C_{\pi}$ for any $\pi \in S_{n}$ because there are ties between candidates. 
In this case, $\mathbf{r}$ lies on one or more of the hyperplanes $H_{ij}=\big\{\mathbf{x}\in\Q^{n}\st x_{i}=x_{j}\big\}$ comprising the 
\emph{braid arrangement}. The sets $C_{\sigma}$ described above are known as \emph{chambers} of the braid arrangement. If we allow for 
nonstrict inequalities in their definition, the resulting objects are called \emph{faces}. The faces of the braid arrangement correspond to ordered 
set partitions of $[n]$ via $G\sim\big(B_{1},\ldots,B_{m}\big)$ when $G$ consists of all $\mathbf{x}\in\mathbb{Q}^{n}$ such that 
$x_{i}>x_{j}$ if and only if there exist $k<\ell$ with $i\in B_{k}$ and $j\in B_{\ell}$. If the results vector $\mathbf{r}$ belongs to $G$, then 
$B_{k}$ is the set of candidates tied for $k^{\text{th}}$ place. The chambers are $n$-dimensional faces representing strict rankings, and two results 
vectors in $\Q^{n}$ lie in the same face if they correspond to identical rankings of the candidates.

Finally, we observe that if $\mathbf{w}\in V_{0}$, then $\sigma\mathbf{w}\in V_{0}$ for all $\sigma\in S_{n}$, so $\mathbf{r}$ as 
defined in equation~\eqref{resultsvector} is a linear combination of sum-zero vectors and thus lies in $V_{0}$ as well. Also, the 
decomposition $Q_{\mathbf{p}}=\sum_{\ell=1}^{n!}p_{\ell}R_{\ell}$ shows that every row and column of $Q_{\mathbf{p}}$ sums 
to $N=\sum_{\ell=1}^{n!}p_{\ell}$, the total number of ballots cast. Thus the condition that $\mathbf{w}\in V_{0}$ is not much of a 
restriction since any $\mathbf{y}\in C_{id}$ can be decomposed as $\mathbf{y}=\overline{\mathbf{y}}+a_{\mathbf{y}}\one$ where 
$\overline{\mathbf{y}}\in W$ and $a_{\mathbf{y}} = \frac{1}{n} \sum_{i=1}^n y_i$. Moreover, $Q_{\mathbf{p}}\mathbf{y}$ lies 
in the same face as $Q_{\mathbf{p}}\overline{\mathbf{y}}$ because
\[ 
Q_{\mathbf{p}}\mathbf{y}
=Q_{\mathbf{p}}\overline{\mathbf{y}}+a_{\mathbf{y}}Q_{\mathbf{p}}\one
=Q_{\mathbf{p}}\overline{\mathbf{y}}+Na_{\mathbf{y}}\one.
\]
Indeed, $\bigcap_{i<j}H_{ij}=\{c\one\st c\in\Q\}$, so it is natural to project the braid arrangement onto the orthogonal complement, 
$V_{0}$. This is called its \emph{essentialization}.


\section{Main Results}
\label{Results}

One of the key insights of this article is that we can construct paradoxical profiles by considering matrices that have all row and column 
sums equal. This is because such a matrix can be shifted and scaled to obtain a \emph{doubly stochastic matrix} (all entries 
nonnegative and all row and column sums equal to $1$). This then enables us to appeal to the \emph{Birkhoff--von Neumann theorem} 
\cite{Birk,vonN}, which states that every doubly stochastic matrix $P$ is a convex combination of permutation matrices, 
$P=\sum_{\ell=1}^{n!}\lambda_{\ell}R_{\ell}$ with $\lambda_{1},\ldots,\lambda_{\ell}\geq 0$ and 
$\lambda _{1}+\cdots+\lambda_{\ell}=1$.

The following result, which is known but included for the sake of completeness, distills these observations into a convenient form 
which will be crucial for subsequent arguments.
\begin{prop}
\label{lin comb} 
Let $\mathcal{P}$ be any collection of $(n-1)^{2}+1$ linearly independent $n\times n$ permutation matrices, and let 
$\mathcal{M}_{n}$ be the vector space of $n\times n$ matrices over $\Q$ with all row and columns sums equal. Then every matrix 
in $\mathcal{M}_{n}$ can be written as a linear combination of matrices in $\mathcal{P}$.
\end{prop}

\begin{proof}
Suppose $S$ is an $n\times n$ matrix with all row and column sums equal to $t$. If $S=\frac{t}{n}J$, then $S$ is a linear combination 
of permutation matrices because $J$ is; see below. Otherwise, let $m=\min _{(i,j)\in[n]^{2}}S(i,j)$. 
Then $P=(t-mn)^{-1}(S-mJ)$ is doubly stochastic, so the Birkhoff--von Neumann theorem shows that $P$ can be 
written as a convex combination of permutation matrices, $P=\sum_{\ell=1}^{n!}\lambda_{\ell}R_{\ell}$. Similarly, 
$\frac{1}{n}J=\sum_{\ell=1}^{n!}\kappa_{\ell}R_{\ell}$ as it too is doubly stochastic. It follows that $S=(t-mn)P+mJ$ is a linear 
combination of permutation matrices. Since every linear combination of permutation matrices has all row and column sums equal, 
$\mathcal{M}$ is precisely the linear span of the permutation matrices. 

To see that the dimension of $\mathcal{M}_{n}$ is $(n-1)^{2}+1$, define $B_{i,j}$ to be the $n\times n$ matrix with $1$'s 
in positions $(i,j)$ and $(n,n)$, $-1$'s in positions $(i,n)$ and $(n,j)$, and $0$'s elsewhere for each $(i,j)\in[n-1]^{2}$. 
If $Z=[z_{i,j}]_{i,j=1}^{n}$ is any matrix with all row and column sums equal to $0$, then it is easy to see that 
$Z=\sum_{i=1}^{n-1}\sum_{j=1}^{n-1}z_{i,j}B_{i,j}$. Now let $S$ be any matrix with all row and column sums equal to $t$. 
Then $S-tI$ has all row and column sums zero; hence $S$ can be expressed as a linear combination of the $B_{i,j}$'s and $I$. 
As these $(n-1)^{2}+1$ matrices are clearly linearly independent, the assertion follows.
\end{proof}

\begin{remark}
Proposition~\ref{lin comb} can also be proved without invoking the Birkhoff--von Neumann theorem by showing that the collection of 
permutation matrices
\[
\mathcal{B}=\big\{R_{(i,j,n)}\st i,j\in[n-1]\text{ are distinct}\big\}\cup\big\{R_{(i,n)}\st i\in[n-1]\big\}\cup\{I\}
\]
is a basis for $\mathcal{M}_{n}$. (The subscripts represent permutations in cycle notation and $R_{\sigma}$ is as previously 
defined.) Linear independence follows by looking at the final rows and columns of the matrices, and 
$\mathcal{M}_{n}=\text{span}(\mathcal{B})$ follows from the dimension argument in the proof of Proposition~\ref{lin comb} upon 
observing that $B_{i,j}=I-R_{(i,n)}-R_{(j,n)}+R_{(j,i,n)}$ for distinct $i,j<n$ and $B_{k,k}=I-R_{(k,n)}$ for $k<n$.
\end{remark}

The following theorem was proved in \cite{DEMO} using facts about the representation theory of $S_{n}$. Our proof is based on 
the same general reasoning---essentially, that one can write $T_{\mathbf{w}}\mathbf{p}=Q_{\mathbf{p}}\mathbf{w}$---but 
uses only linear algebra. In words, one can fix in advance a number of different positional voting procedures, along with desired election 
outcomes for each procedure, and then find (infinitely many) profiles such that each procedure yields the corresponding outcome!
\begin{theorem}
\label{thm:main} \textnormal{(\cite[Theorem 1]{DEMO})} 
Given any linearly independent weighting vectors $\mathbf{w}_{1},\ldots,\mathbf{w}_{n-1}\in W$ and any 
results vectors $\mathbf{r}_{1},\ldots,\mathbf{r}_{n-1}\in V_{0}$, there are infinitely many profiles $\mathbf{p}\in\Q^{n!}$ with 
$T_{\mathbf{w}_{k}}\mathbf{p}=\mathbf{r}_{k}$ for $k=1,\ldots,n-1$. 
\end{theorem}

\begin{proof}
The general strategy will be to construct a matrix $Q$ such that $Q\mathbf{w}_{k}=\mathbf{r}_{k}$ for $k=1,\ldots,n-1$ and show that this 
matrix has right and left eigenvectors $\one$ and $\one^{\T}$; hence its row and column sums are constant. Proposition~\ref{lin comb} 
then gives $Q=\sum_{\ell}p_{\ell}R_{\ell}$, and thus $T_{\mathbf{w}_{k}}\mathbf{p}=Q\mathbf{w}_{k}=\mathbf{r}_{k}$.

To begin, define $\mathbf{r}_{0}=\mathbf{w}_{0}=\one$ and set 
\[
F=\big[\,\setlength\arraycolsep{3pt}\begin{matrix} \mathbf{w}_{0} & \mathbf{w}_{1}
 & \cdots & \mathbf{w}_{n-1}\end{matrix}\,\big],\:\, 
R=\big[\,\setlength\arraycolsep{3pt}\begin{matrix} \mathbf{r}_{0}
 & \mathbf{r}_{1} & \cdots & \mathbf{r}_{n-1}\end{matrix}\,\big],\;\text{ and }\; 
Q=RF^{-1}.
\]
($F$ is invertible because $\mathbf{w}_{1},\ldots,\mathbf{w}_{n-1}$ are linearly independent and all orthogonal 
to $\mathbf{w}_{0}$.) Then $Q\mathbf{w}_{k}=\mathbf{r}_{k}$ for $k=0,\ldots,n-1$ since
\[
\setlength\arraycolsep{3pt}\big[\,\begin{matrix}
Q\mathbf{w}_{0} & Q\mathbf{w}_{1} & \cdots & Q\mathbf{w}_{n-1}\end{matrix}\,\big]=QF=R
=\setlength\arraycolsep{3pt}\big[\,\begin{matrix}
\mathbf{r}_{0} & \mathbf{r}_{1} & \cdots & \mathbf{r}_{n-1}\end{matrix}\,\big].
\]

The condition $Q\mathbf{w}_{0}=\mathbf{r}_{0}$ implies that the rows of $Q$ sum to $1$. To see that the columns sum to $1$, 
we first observe that 
\[
\mathbf{w}_{0}^{\T}R=\setlength\arraycolsep{3pt}\big[\,\begin{matrix}
\left\langle \mathbf{w}_{0},\mathbf{r}_{0}\right\rangle  & \left\langle \mathbf{w}_{0},\mathbf{r}_{1}\right\rangle 
 & \cdots & \left\langle \mathbf{w}_{0},\mathbf{r}_{n-1}\right\rangle \end{matrix}\,\big]
=\setlength\arraycolsep{3pt}[\,\begin{matrix} n & 0 & \cdots & 0\end{matrix}\,]
\]
since $\mathbf{w}_{0}=\mathbf{r}_{0}=\one$ is orthogonal to each of $\mathbf{r}_{1},\ldots,\mathbf{r}_{n-1}$ by assumption. 
As such, 
\[
\mathbf{w}_{0}^{\T}Q=\mathbf{w}_{0}^{\T}RF^{-1}
=\setlength\arraycolsep{3pt}[\,\begin{matrix} n & 0 & \cdots & 0\end{matrix}\,]F^{-1}
=n\bm{f}
\]
where $\bm{f}$ is the first row of $F^{-1}$. Since $F^{-1}F=I$, we must have  
$\left\langle \bm{f}^{\T},\mathbf{w}_{0}\right\rangle =1$ and 
$\left\langle \bm{f}^{\T},\mathbf{w}_{k}\right\rangle =0$ for $k=1,\ldots,n-1$. 
The latter condition implies that $\bm{f}=C\one^{\T}$ for some $C$, so the former implies 
$1=\left\langle \bm{f}^{\T},\mathbf{w}_{0}\right\rangle =C\left\langle \one,\one\right\rangle =nC$. 
Therefore, 
\[
\mathbf{w}_{0}^{\T}Q=n\bm{f}=\mathbf{w}_{0}^{\T}, 
\]
so the columns of $Q$ sum to $1$ as well. 

Since $Q$ has all rows and columns summing to $1$, it follows from Proposition~\ref{lin comb} that it is a linear combination of 
permutation matrices. In other words, there exists $\mathbf{p}\in\Q^{n!}$ such that $Q=\sum_{\ell=1}^{n!}p_{\ell}R_{\ell}$. 
Accordingly, $T_{\mathbf{w}_{k}}\mathbf{p}=Q\mathbf{w}_{k}=\mathbf{r}_{k}$ for $k=1,\ldots,n-1$. In fact there are infinitely 
many such $\mathbf{p}$ since there are $n!$ permutation matrices and the space of doubly stochastic matrices is 
$(n^{2}-2n+2)$-dimensional.
\end{proof}

The preceding proof works just as well if one takes the weighting vectors to lie in $\overline{W}$, the closure of $W$. 
This allows for voting schemes in which the same point value can be assigned to multiple candidates, as in the ``vote for your favorite $k$'' 
system given by $\mathbf{w}
=[\,\setlength\arraycolsep{3pt}\begin{matrix}n-k & \cdots & n-k & -k & \cdots & -k\end{matrix}\,]^{\T}$.
One may impose additional constraints such as all weighting vectors having the same positions tied by restricting to some lower-dimensional 
face $G\subset\overline{W}$, but then the linear independence condition dictates that there are only $d=\dim(G)$ weighting/results 
vectors. To treat this case, take $\mathbf{w}_{1},\ldots,\mathbf{w}_{d}$ to be linearly independent vectors in $G$ and 
$\mathbf{r}_{1},\ldots,\mathbf{r}_{d}$ to be the desired results vectors in $V_{0}$. Then choose 
$\mathbf{w}_{d+1},\ldots,\mathbf{w}_{n-1}$ to be any vectors in $\overline{W}$ for which 
$\mathbf{w}_{1},\ldots,\mathbf{w}_{n-1}$ are linearly independent and let $\mathbf{r}_{d+1},\ldots,\mathbf{r}_{n-1}$ 
be any results vectors in $V_{0}$.

Also, observe that $Q=RF^{-1}$ is explicit and may be easily realized as a linear combination of doubly stochastic matrices. (The 
entries of $Q$ may be negative, so one must add an appropriate multiple of the all ones matrix and rescale to obtain a doubly 
stochastic matrix $P$ as in the proof of Proposition~\ref{lin comb}.) As there are algorithms for finding a Birkhoff--von Neumann 
decomposition of any doubly stochastic matrix \cite{DufUcar}, our method actually provides a construction of the paradoxical profile.

\begin{example}
Suppose that 
\begin{equation*}
\mathbf{w}_{1}=\begin{bmatrix} 3 \\ 1 \\ -1 \\ -3 \end{bmatrix}, \mathbf{w}_{2}=\begin{bmatrix} 1 \\ 1 \\ 1 \\ -3 \end{bmatrix}, 
\mathbf{w}_{3}=\begin{bmatrix} 17 \\ 1 \\ -7 \\ -11 \end{bmatrix},
\mathbf{r}_{1}=\begin{bmatrix} -2 \\ -11 \\ 4 \\ 9 \end{bmatrix}, \mathbf{r}_{2}=\begin{bmatrix} 4 \\ 5 \\ 3 \\ -12 \end{bmatrix}, 
\mathbf{r}_{3}=\begin{bmatrix} 13 \\ -2 \\ -6 \\ -5 \end{bmatrix}.
\end{equation*}
Then 
\begin{equation*}
Q=\begin{bmatrix} 1 & -2 & 4 & 13\\ 1 & -11 & 5 & -2\\ 1 & 4 & 3 & -6\\ 1 & 9 & -12 & -5\end{bmatrix} 
    \begin{bmatrix} 1 & 3 & 1 & 17\\ 1 & 1 & 1 & 1\\ 1 & -1 & 1 & -7 \\ 1 & -3 & -3 & -11\end{bmatrix}^{\displaystyle{-1}}
=\dfrac{1}{8} \begin{bmatrix} 27 & -64 & 51 & -6\\ 49 & -146 & 113 & -8\\
 -17 & 50 & -21 & -4 \\ -51 & 168 & -135 & 26\end{bmatrix}
\end{equation*}
and 
\[
P=\big(1+4\cdot\tfrac{146}{8}\big)^{-1}\big(Q+\tfrac{146}{8} J\big)
=\frac{1}{592} \begin{bmatrix} 173 & 82 & 197 & 140 \\ 195 & 0 & 259 & 138 \\ 129 & 196 & 125 & 142 
 \\ 95 & 314 & 11 & 172\end{bmatrix}
\]
is doubly stochastic. 

Using built-in functionality in the computer algebra system \emph{SageMath} \cite{sage}, a Birkhoff--von Neumann decomposition of $P$ is given by 
\begin{align*}
P & =\tfrac{71}{296}R_{(1,4,2,3)}+\tfrac{31}{592}R_{(1,4,3,2)}+\tfrac{43}{148}R_{(2,3,1,4)}+\tfrac{11}{592}R_{(2,3,4,1)}
+\tfrac{3}{148}R_{(2,4,3,1)}\\
 & \qquad +\tfrac{3}{148}R_{(3,4,1,2)}+\tfrac{117}{592}R_{(3,4,2,1)}+\tfrac{41}{296}R_{(4,1,3,2)}+\tfrac{13}{592}R_{(4,3,1,2)}.
\end{align*}
Here the subscripts represent permutations in one-line notation.

Since $Q=74P-\tfrac{73}{4}J$ and $J=R_{(1,2,3,4)}+R_{(2,1,4,3)}+R_{(3,4,1,2)}+R_{(4,3,2,1)}$, we see that 
\begin{align*}
Q & = -\tfrac{73}{4}R_{(1,2,3,4)}+\tfrac{71}{4}R_{(1,4,2,3)}+\tfrac{31}{8}R_{(1,4,3,2)}-\tfrac{73}{4}R_{(2,1,4,3)}
+\tfrac{43}{2}R_{(2,3,1,4)}\\
 & \qquad +\tfrac{11}{8}R_{(2,3,4,1)}+\tfrac{3}{2}R_{(2,4,3,1)}-\tfrac{67}{4}R_{(3,4,1,2)}+\tfrac{117}{8}R_{(3,4,2,1)}
+\tfrac{41}{4}R_{(4,1,3,2)}\\
 & \qquad\qquad+\tfrac{13}{8}R_{(4,3,1,2)}-\tfrac{73}{4}R_{(4,3,2,1)}.
\end{align*}
Thus one profile for which the weight $\mathbf{w}_{k}$ produces the result $\mathbf{r}_{k}$ consists of $-\tfrac{73}{4}$ votes for 
$1$ above $2$ above $3$ above $4$, $\frac{71}{4}$ votes for $1$ above $4$ above $2$ above $3$, and so forth. 
\end{example}

In typical settings, we are concerned with ordinal rather than cardinal rankings, and Theorem~\ref{thm:main} can then be used to 
generate significantly many more outcomes. To facilitate the ensuing argument, we record the following simple lemma.
\begin{lemma}
\label{scaling lemma} 
For any $\mathbf{w}\in W$, $\mathbf{x}\in V_{0}$, there is some rational number $\eta_{0}>0$ such that 
$\eta\mathbf{w}+\mathbf{x}\in W$ for all $\eta\geq\eta_{0}$.
\end{lemma}

\begin{proof}
Let $m=\min_{1\leq k\leq n-1}(w_{k}-w_{k+1})$ and $M=\max_{1\leq k\leq n}\left|x_{k}\right|$, and set $\eta=3M/m$. 
Then the successive entries of $\eta\mathbf{w}$ differ by at least $3M$, so adding $\mathbf{x}$ does not change their relative 
order.
\end{proof}

Also, recall that a conical combination of vectors is a linear combination in which all coefficients are nonnegative, and observe 
that $W$ is closed under nontrivial conical combinations: if $\mathbf{w}_{1}, \mathbf{w}_{2}, \ldots, \mathbf{w}_{k} \in W$, then 
$ c_{1}\mathbf{w}_{1}+c_{2} \mathbf{w}_{2} + \cdots + c_{k} \mathbf{w}_{k} \in W$ whenever $c_{1}, \ldots, c_{k} \geq 0$ with 
$c_{i} \neq 0$ for at least one $i$. A set with this property is called a \emph{convex cone}.

Our next theorem implies the result from \cite{Saari} that there exist profiles from which $\big(\frac{n-1}{n}\big)n!$ different 
ordinal rankings can be obtained by judicious choices of weighting vectors in $W$.

\begin{theorem}
\label{number rankings lower bound} 
There exist infinitely many profiles $\mathbf{p}\in\Q^{n!}$ such that for every $\pi\in S_{n}$ satisfying 
$\pi(n)\neq1$, there is some $\mathbf{w}(\pi)\in W$ with $T_{\mathbf{w}(\pi)}\mathbf{p}\in C_{\pi}$.
\end{theorem}

\begin{proof}
To begin, let $\mathbf{w}_{1},\ldots,\mathbf{w}_{n-1}$ be any linearly independent vectors in $W$. By 
Lemma~\ref{scaling lemma}, we may scale $\mathbf{w}_{1}$ so that $\mathbf{w}_{1}-n\binom{n+1}{2}\mathbf{w}_{k}\in W$ 
for $k=2,\ldots,n-1$. (This will be important later.) Set $\mathbf{f}_{k}=\mathbf{e}_{k}-\mathbf{e}_{k+1}$ for $k=1,\ldots,n-1$ where 
$\mathbf{e}_{1},\ldots,\mathbf{e}_{n}$ are the standard basis vectors in $\Q^{n}$, and let $\mathbf{p}$ be such that 
$T_{\mathbf{w}_{k}}\mathbf{p} = \mathbf{f}_{k}$ for each $k$. There are infinitely many such $\mathbf{p}$ by Theorem~\ref{thm:main}.

Fix $\pi\in S_{n}$ with $\pi(n) \neq1$. The result will follow if we can find $\alpha_{1},\ldots,\alpha_{n-1}$ so that 

\vspace{-.6cm}

\[
\mathbf{w}=\mathbf{w}(\pi)=\sum_{k=1}^{n-1}\alpha_{k}\mathbf{w}_{k}
\]
belongs to $W$ and

\vspace{-.6cm}

\[
\mathbf{s}=\mathbf{s}(\pi)=T_{\mathbf{w}}\mathbf{p}=Q_{\mathbf{p}}\mathbf{w}=\sum_{k=1}^{n-1}\alpha_{k}\mathbf{f}_{k}
\]
belongs to $C_{\pi}$.

Now for $1\leq i<j\leq n$, define $\mathbf{f}_{ij}=\mathbf{e}_{i}-\mathbf{e}_{j}=\sum_{k=i}^{j-1}\mathbf{f}_{k}$ and 
$\mathbf{w}_{ij}=\sum_{k=i}^{j-1}\mathbf{w}_{k}$. The latter are contained in $W$ since it is a convex cone. For any collection of numbers 
$\{\beta_{ij}\}_{i<j}$, if we define $\alpha_{k}=\sum_{i=1}^k \sum_{j=k+1}^n \beta_{ij}$, then we have
\[
\sum_{i<j}\beta_{ij}\mathbf{f}_{ij}=\sum_{i<j}\beta_{ij}\sum_{k=i}^{j-1}\mathbf{f}_{k}
=\sum_{k=1}^{n-1}\alpha_{k}\mathbf{f}_{k}
\]
and 

\vspace{-.6cm}

\[
\sum_{i<j}\beta_{ij}\mathbf{w}_{ij}=\sum_{k=1}^{n-1}\alpha_{k}\mathbf{w}_{k}.
\]

Accordingly, it suffices to construct $\mathbf{w}=\sum_{i<j}\beta_{ij}\mathbf{w}_{ij}\in W$ with 
$\mathbf{s}=Q_{\mathbf{p}}\mathbf{w}=\sum_{i<j}\beta_{ij}\mathbf{f}_{ij}\in C_{\pi}$.
Note that each $\mathbf{w}_{ij}$ is in $W$, so $\mathbf{w}$ will be as well whenever the $\beta_{ij}$'s are 
nonnegative (and not all $0$).

First consider the case in which $\pi(n)=n$. Then we can take $\beta_{kn}=n-\pi^{-1}(k)$ for $k=1,\ldots,n-1$ 
and $\beta_{ij}=0$ for $j\neq n$. As all $\beta_{ij}$ are nonnegative, $\mathbf{w}=\sum_{i<j}\beta_{ij}\mathbf{w}_{ij}\in W$. To see that 
$\mathbf{s} \in C_{\pi}$, we observe that $s_k = n-\pi^{-1}(k)$ for $k=1, \ldots, n-1$ and $s_n = -\sum_{k=1}^{n}s_{k}$. This gives the $k^{\text{th}}$ place candidate 
$s_{\pi(k)}=n-\pi^{-1}(\pi(k))=n-k>0$ points for $k=1,\ldots,n-1$ and gives $-\binom{n}{2}<0$ points to 
candidate $n$. 

If $\pi(n) = b$ with $1<b<n$, let $\widetilde{\pi}$ be the permutation formed from $\pi$ by moving $n$ to last place (so 
$\widetilde{\pi}(i)=\pi(i)$ for $i<\pi^{-1}(n)$, $\widetilde{\pi}(i)=\pi(i+1)$ for $\pi^{-1}(n)\leq i<n$, 
and $\widetilde{\pi}(n)=n$), and let $\widetilde{\mathbf{w}}=\mathbf{w}(\widetilde{\pi})$, 
$\widetilde{\mathbf{s}}=\mathbf{s}(\widetilde{\pi})$ be constructed as above. Now set 
$\mathbf{w}=\widetilde{\mathbf{w}}-\gamma_{bn}\mathbf{w}_{bn}$ where

\vspace{-.6cm}

\[
\gamma_{bn}=\binom{n}{2}+n-\pi^{-1}(n)+\frac{1}{2}.
\] 
Then $\mathbf{s}=Q_{\mathbf{p}}\mathbf{w}=\widetilde{\mathbf{s}}-\gamma_{bn}\mathbf{f}_{bn}$, and we will be done upon establishing 
that $\mathbf{s} \in C_{\pi}$ and $\mathbf{w} \in W$. 

We first note that $\widetilde{\mathbf{s}} \in C_{\widetilde{\pi}}$, and $\mathbf{s}$ differs 
from $\widetilde{\mathbf{s}}$ by adding $\gamma_{bn}$ in position $n$ and subtracting $\gamma_{bn}$ in position $b$. As a result,
candidate $n$ now has $n-\pi^{-1}(n)+\frac{1}{2}$ points, candidate $b$ has a negative number of points, and all 
other candidates have the same point values as in $\widetilde{\mathbf{s}}$. It follows that $\mathbf{s} \in C_\pi$.

Now we turn our attention to $\mathbf{w}$, showing that it can be written as a linear combination of vectors in $W$ with nonnegative coefficients. 
Recall that $\mathbf{w}_{1}$ is scaled so that $\mathbf{w}_{1}-n\binom{n+1}{2}\mathbf{w}_{k}\in W$ for $k=2,\ldots,n-1$. Also, since  $1<b<n$, 
we have
\[ 
(n-b)\gamma_{bn}=(n-b)\left[\binom{n+1}{2}-\pi^{-1}(n)+\frac{1}{2}\right]<n\binom{n+1}{2}.
\]  
Thus for each $1<k<n$, $\mathbf{w}_{1}-(n-b)\gamma_{bn}\mathbf{w}_{k}\in W$ since it is obtained by adding a 
positive multiple of $\mathbf{w}_{k}$ to  $\mathbf{w}_{1}-n\binom{n+1}{2}\mathbf{w}_{k}$.
This shows that
\[
\mathbf{w}_{1}-\gamma_{bn}\mathbf{w}_{bn}=\mathbf{w}_{1}-\gamma_{bn}\sum_{k=b}^{n-1}\mathbf{w}_{k}
=\frac{1}{n-b}\sum_{k=b}^{n-1}\Big(\mathbf{w}_{1}-(n-b)\gamma_{bn}\mathbf{w}_{k}\Big)\in W.
\]
To complete the proof, note that $\widetilde{\mathbf{w}}=\sum_{k=1}^{n-1}\big(n-\widetilde{\pi}^{-1}(k)\big)\mathbf{w}_{k}$ 
and $n-\widetilde{\pi}^{-1}(k)\geq1$ for all $k<n$, so 
\begin{align*}
\mathbf{w} & =\widetilde{\mathbf{w}}-\gamma_{bn}\mathbf{w}_{bn}\\
 & =(n-\widetilde{\pi}^{-1}(1)-1)\mathbf{w}_{1}
+\sum_{k=2}^{n-1}\big(n-\widetilde{\pi}^{-1}(k)\big)\mathbf{w}_{k}
+\big(\mathbf{w}_{1}-\gamma_{bn}\mathbf{w}_{bn}\big)
\end{align*}
is a conical combination of vectors in $W$ and so belongs to $W$. 
\end{proof}

\begin{remark}
The $\pi(n)=n$ part of the above argument can be interpreted as saying that if we have $n-1$ ``serious 
candidates,'' then by introducing ``dummy candidate'' $n$, who is assured to lose, there are profiles that 
achieve any relative ordering of the serious candidates by choosing appropriate weighting vectors. 

Note that one can always left-multiply $Q_{\mathbf{p}}$ by a permutation matrix to relabel the candidates, so there is 
nothing special about $n$.
\end{remark}

It is often more convenient to work with $Q_{\mathbf{p}}$ than $T_{\mathbf{w}}$, so we take a moment to observe 
that if one only cares about the ordinal rankings of candidates, then it can always be assumed that each 
profile consists of nonnegative integers or that the matrix $Q_{\mathbf{p}}$ is doubly stochastic.

\begin{prop}
\label{convenient p} 
Using the notation from above,
\begin{enumerate}
\item For any $\mathbf{p}\in\Q^{n!}$, there exists $\widehat{\mathbf{p}}\in\N_{0}^{n!}$ with 
$Q_{\widehat{\mathbf{p}}}\mathbf{w}$ lying in the same face as $Q_{\mathbf{p}}\mathbf{w}$ for all 
$\mathbf{w}\in \overline{W}$.

\item For any $\mathbf{p}\in\Q^{n!}$, there exists $\widetilde{\mathbf{p}}\in\Q^{n!}$ such that 
$\widetilde{p}_{\ell}\geq 0$ for all $\ell$, $\sum_{\ell=1}^{n!}\widetilde{p}_{\ell}=1$, and 
$Q_{\widetilde{\mathbf{p}}}\mathbf{w}$ lies in the same face as $Q_{\mathbf{p}}\mathbf{w}$ for all 
$\mathbf{w}\in \overline{W}$.
\end{enumerate}
\end{prop}

\begin{proof}
For the first claim, set $x=\min_{\ell}p_{\ell}$, let $d$ be the least common denominator 
of $p_{1},\ldots,p_{n!}$, and define $\widehat{\mathbf{p}}=d(\mathbf{p}-x\one)\in\N_{0}^{n!}$.
If $Q_{\mathbf{p}}\mathbf{w}=T_{\mathbf{w}}\mathbf{p}=\mathbf{r}$, then 
\[
Q_{\widehat{\mathbf{p}}}\mathbf{w}=T_{\mathbf{w}}d(\mathbf{p}-x\one)
=dT_{\mathbf{w}}\mathbf{p}-dxT_{\mathbf{w}}\one=d\mathbf{r}
\]
since $T_{\mathbf{w}}\one=\zero$. Thus $Q_{\widehat{\mathbf{p}}}\mathbf{w}$ is a positive multiple of 
$Q_{\mathbf{p}}\mathbf{w}$ and so lies in the same face.

For the second claim, set $m=\min_{i,j}Q_{\mathbf{p}}(i,j)$, $c=\big(\sum_{\ell=1}^{n!}p_{\ell}-mn\big)^{-1}>0$, 
and $\widetilde{Q}_{\mathbf{p}}=c(Q_{\mathbf{p}}-mJ)$ where $J$ is the all ones matrix. (Here we are assuming that 
$Q_{\mathbf{p}}$ does not have all entries equal. If it does, one can just take $\widetilde{Q}_{\mathbf{p}}=\tfrac{1}{n}J$.) 
Then $\widetilde{Q}_{\mathbf{p}}$ is nonnegative with rows and columns summing to $1$, so the Birkhoff--von Neumann theorem 
guarantees the existence of a nonnegative $\widetilde{\mathbf{p}}\in\Q^{n!}$ with entries summing to $1$ that satisfies 
$Q_{\widetilde{\mathbf{p}}}=\widetilde{Q}_{\mathbf{p}}$. The claim follows since
\[
\widetilde{Q}_{\mathbf{p}}\mathbf{w}=c(Q_{\mathbf{p}}-mJ)\mathbf{w}
=cQ_{\mathbf{p}}\mathbf{w}-mcJ\mathbf{w}=cQ_{\mathbf{p}}\mathbf{w}
\]
lies in the same face as $Q_{\mathbf{p}}\mathbf{w}$.
\end{proof}

The first part of Proposition~\ref{convenient p} is only helpful inasmuch as one might object to having negative or fractional 
votes cast. The second part is useful from a more mathematical perspective: Since $W$ is a convex cone, its image under 
$Q_{\mathbf{p}}$ is the same as its image under the doubly stochastic matrix $Q_{\widetilde{\mathbf{p}}}$. 
Thus we can study possible ordinal outcomes of positional voting procedures by looking at the image of $W$ (or its closure 
$\overline{W}$) under multiplication by doubly stochastic matrices. An example of the utility of this observation is that, by 
Proposition~\ref{lin comb}, any set of $(n-1)^{2}+1$ linearly independent permutation matrices can give rise to all paradoxical profiles 
for positional voting procedures. In other words, one needs only this many distinct preferences among the electorate. In fact, the construction 
from Proposition~\ref{convenient p} shows that one may take the doubly stochastic matrix to have at least one entry equal to zero (or all 
entries equal), and a result from \cite{Brua} then implies that a Birkhoff--von Neumann decomposition of size at most $(n-1)^{2}$ exists. 
This is the content of Theorem 4 in \cite{Saari}.

The final result of this section uses the doubly stochastic matrix perspective to show that a profile can give rise to at most 
$\frac{n-1}{n}n!$ strict societal rankings; Theorem~\ref{number rankings lower bound} shows that this is sharp. The result was 
first proved in \cite{Saari} by analyzing certain simplicial constructs geometrically. Our strategy is to show that the set of possible 
societal rankings for any profile lies in a closed half-space whose boundary contains the origin, and then argue that the complement 
of this half-space properly contains at least $(n-1)!$ of the $n!$ chambers $C_\pi$ (corresponding to impossible rankings).

\begin{theorem}
\label{number rankings upper bound} \textnormal{(\cite[Theorem 3a]{Saari})} 
For any $\mathbf{p}\in\Q^{n!}$, there are at most $n!-(n-1)!$ permutations $\pi\in S_{n}$ such that 
$T_{\mathbf{w}}\mathbf{p}\in C_{\pi}$ for some $\mathbf{w}\in\overline{W}$.
\end{theorem}

\begin{proof}
Given any profile $\mathbf{p}\in\Q^{n!}$, there is a doubly stochastic matrix $Q_{\widetilde{\mathbf{p}}}$ such that the possible 
strict societal rankings arising from $\mathbf{p}$ are precisely those $\pi\in S_{n}$ for which 
$C_{\pi}\cap Q_{\widetilde{\mathbf{p}}}\overline{W}\neq\emptyset$. Equivalently, since doubly stochastic matrices map sum-zero 
vectors to sum-zero vectors, writing $T:V_{0}\rightarrow V_{0}$ for the linear map defined by 
$T(\mathbf{v})=Q_{\widetilde{\mathbf{p}}}\mathbf{v}$ and denoting $\widetilde{C}_{\pi}=C_{\pi}\cap V_{0}$, $\pi$ is a possible 
ranking for $\mathbf{p}$ if and only if $\widetilde{C}_{\pi}\cap T\left(\overline{W}\right)\neq\emptyset$. 

Now $T\left(\overline{W}\right)$ is the linear image of a convex set and thus is convex. Also, $T\left(\overline{W}\right)$ is a 
proper subset of $V_{0}$ because if $T$ is not bijective then $T\left(\overline{W}\right)\subseteq T\left(V_{0}\right)\subset V_{0}$, 
and if $T$ is bijective then 
\[
T\left(\overline{W}\right)\cap-T\left(\overline{W}\right)
=T\left(\overline{W}\right)\cap T\left(-\overline{W}\right)\subseteq T(\overline{W}\cap-\overline{W})
=\{\zero\}.
\] 
This means that there is some $\mathbf{v}\in V_{0}\setminus T\left(\overline{W}\right)$, hence 
$\e \mathbf{v}\in V_{0}\setminus T\left(\overline{W}\right)$ for every $\e>0$, so $\zero\in\partial\,T\left(\overline{W}\right)$. 
Accordingly, the supporting hyperplane theorem shows that there is a hyperplane $H$ through the origin in $V_{0}$ with 
$T\left(\overline{W}\right)$ contained entirely in the associated closed positive half-space \cite{Lang}. 
In particular, there is a nonzero $\mathbf{h}\in V_{0}$ such that $\pi$ is a possible ranking for $\mathbf{p}$ only if there is an 
$\mathbf{r}\in \widetilde{C}_{\pi}$ such that $\langle \mathbf{h},\mathbf{r}\rangle \geq 0$.

We will establish the result by showing that this constraint precludes $(n-1)!$ rankings from being possible outcomes associated with 
$\mathbf{p}$. To this end, oberve that $\left\langle \mathbf{h},\mathbf{r}\right\rangle <0$ for all $\mathbf{r}\in \widetilde{C}_{\pi}$ 
if and only if $\left\langle \pi\mathbf{h},\mathbf{w}\right\rangle <0$ for all $\mathbf{w}\in W$ where 
$\pi\mathbf{h}=[\,\setlength\arraycolsep{3pt}\begin{matrix} h_{\pi^{-1}(1)} & \cdots & h_{\pi^{-1}(n)}\end{matrix}\,]^\T$, 
so it suffices to show that there are at least $(n-1)!$ ways to permute the entries of $\mathbf{h}$ so that it has negative inner product with
every vector in $W$. 

Now if $\mathbf{v},\mathbf{w}\in V_{0}$, then 
\begin{equation}
\label{eq: inner product with partial sums}
\begin{aligned}
\hspace{-.15cm}\left\langle \mathbf{v},\mathbf{w}\right\rangle  & =\sum_{k=1}^{n}v_{k}w_{k}
=v_{1}(w_{1}-w_{2})+(v_{1}+v_{2})(w_{2}-w_{3})\\
 & \:\;+(v_{1}+v_{2}+v_{3})(w_{3}-w_{4})+\cdots+(v_{1}+\cdots+v_{n-1})(w_{n-1}-w_{n}), 
\end{aligned}
\end{equation}
as there is cancellation among successive terms and $-w_{n}(v_{1}+\cdots+v_{n-1})=w_{n}v_{n}$.
If $\mathbf{w}\in W$, then $w_{k}-w_{k+1}>0$ for each $k=1,\ldots,n-1$, so \eqref{eq: inner product with partial sums} shows that
$\left\langle \mathbf{v},\mathbf{w}\right\rangle <0$ whenever the partial sums satisfy $\sum_{k=1}^{m}v_{k}\leq0$ for $m=1,\ldots,n-1$. 
(At least one of these inequalities would have to be strict for $\mathbf{v}\in V_{0}\setminus\{\zero\}$.)

Combining these observations, we conclude that there are at least $(n-1)!$ impossible rankings as long as there are $(n-1)!$ ways to permute 
the entries of $\mathbf{h}$ so that the partial sums are all nonpositive. To see that this is so, let $\pi'\in S_{n-1}$ and define $\pi\in S_{n}$ 
by $\pi(k)=\pi'(k)$ for $k<n$ and $\pi(n)=n$. Define $\pi_{m}\in S_{n}$ by $\pi_{m}(k)=\pi(k+m)$ for $k=1,\ldots,n$ and $m=0,1,\ldots,n-1$, 
where the addition in the argument is performed modulo $n$. The assertion will follow if we can show that at least one of these cyclic 
shifts of $\pi$ has the property that $s_{k}(m):=\sum_{j=1}^{k}h_{\pi_{m}(j)}\leq0$ for $k=1,\ldots,n-1$.  

For this, choose $m$ so that $s_{m}(0)=\sum_{j=1}^{m}h_{\pi(j)}$ is maximal. We claim that $s_{k}(m)\leq0$ for all $1\leq k<n$. 
Indeed, for $1\leq k\leq n-m$, 
\[
s_{k}(m)=\sum_{j=m+1}^{m+k}h_{\pi(j)}\leq 0
\]
by maximality of $s_{m}(0)$. Also, $\mathbf{h}\in V_{0}$ implies that $\sum_{j=m+1}^{n}h_{\pi(j)}=-s_{m}(0)$, so for 
$n-m<\ell\leq n-1$, we have 
\[
s_{\ell}(m)=\sum_{j=m+1}^{n}h_{\pi(j)}+\sum_{i=1}^{\ell+m+1-n}h_{\pi(i)}\leq\sum_{j=m+1}^{n}h_{\pi(j)}+s_{m}(0)=0.
\]
This completes the claim and the proof.
\end{proof}

\begin{remark}
\label{rem: hyperplane}
Demonstrating that at least $(n-1)!$ chambers lie on the opposite side of $H$ from $T\left(\overline{W}\right)$ was a bit involved because 
we sought to keep the discussion self-contained. However, if one brings the full power of the theory of hyperplane arrangements to 
bear on the problem (which is one of the advantages of introducing this perspective), then much more efficient arguments are possible. 
For instance, Corollary 14.2 and equation (6.9) in \cite{AguMah} give the number of chambers in a \emph{generic half-space} of the braid 
arrangement---that is, a half-space defined by a hyperplane through the origin that does not contain any one-dimensional faces---as $(n-1)!$. 
Since $H$ can be perturbed so as to be generic and this can only decrease the number of chambers properly contained on either side, the 
desired inequality follows immediately.
\end{remark}


\section{Choosing Weighting Vectors}
\label{Weighting Vectors}

In this final section, we characterize the possible results vectors arising from a given profile $\mathbf{p}$ as the conical hull of a 
set of vectors constructed from the columns of $Q_{\mathbf{p}}$. The construction is surprisingly simple and provides a 
straightforward procedure to reverse engineer an election by selecting desirable weights for a given profile.

We begin by providing a convenient description of the space of (nonstrict) weighting vectors $\overline{W}$. Define 
$\mathbf{v}_{1},\ldots,\mathbf{v}_{n-1}\in\Q^{n}$ by $\mathbf{v}_{k}=\frac{k}{n}\one-\sum_{j=n-k+1}^{n}\mathbf{e}_{j}$, 
so that 
\begin{align*}
\mathbf{v}_{1} & =  \big[\,\setlength\arraycolsep{3pt}\begin{matrix}
\frac{1}{n} & \cdots & \frac{1}{n} & -\frac{n-1}{n}
\end{matrix}\,\big]^{\T}\\
\mathbf{v}_{2} & =\big[\,\setlength\arraycolsep{3pt}\begin{matrix}
\frac{2}{n} & \cdots & \frac{2}{n} & -\frac{n-2}{n} & -\frac{n-2}{n}
\end{matrix}\,\big]^{\T}\\
 & \qquad\qquad\qquad\quad \vdots \\
\mathbf{v}_{n-2} & =\big[\,\setlength\arraycolsep{3pt}\begin{matrix}
\frac{n-2}{n} & \frac{n-2}{n} & -\frac{2}{n} & \cdots &  -\frac{2}{n}
\end{matrix}\,\big]^{\T}\\
\mathbf{v}_{n-1} & =\big[\,\setlength\arraycolsep{3pt}\begin{matrix}
\frac{n-1}{n} & -\frac{1}{n} & \cdots &  -\frac{1}{n}
\end{matrix}\,\big]^{\T}.
\end{align*}

\begin{prop}
\label{basis}
Let $\mathbf{v}_{1},\ldots,\mathbf{v}_{n-1}$ be as above. Then
\[
\overline{W}=\big\{c_{1}\mathbf{v}_{1}+\cdots+c_{n-1}\mathbf{v}_{n-1}\st c_{1},\ldots,c_{n-1}\geq 0\big\}.
\]
\end{prop}

\begin{proof}
Clearly $\mathbf{v}_{k}\in\overline{W}$ for $k=1,\ldots,n-1$, and thus so is any conical combination thereof. 
Conversely, given any $\mathbf{w} \in\overline{W}$, we have
\[ 
\mathbf{w}= \sum_{k=1}^{n-1} a_{k} \mathbf{v}_{k},
\]
where
\[
a_{k} = w_{n-k}-w_{n-k+1}\geq 0. 
\]
Indeed, the $i^{\text{th}}$ coordinate of $\sum_{k=1}^{n-1} a_{k} \mathbf{v}_{k}$ is
\begin{align*}
\frac{1}{n}\sum_{j=1}^{n-i}ja_{j} &  -  \frac{1}{n}\sum_{j=n-i+1}^{n-1}a_{j}(n-j)\\
 & = \frac{1}{n}\sum_{j=1}^{n-1}j(w_{n-j}-w_{n-j+1})-\sum_{j=n-i+1}^{n-1}(w_{n-j}-w_{n-j+1}) \\
 & = \frac{1}{n}\Big((n-1)w_{1}-\sum_{k=2}^{n}w_{k}\Big) - \sum_{k=1}^{i-1}(w_{k}-w_{k+1}) \\
 & = \frac{1}{n}\Big(nw_{1}-\sum_{k=1}^{n}w_{k}\Big) - (w_{1}-w_{i})=w_{1}-(w_{1}-w_{i})=w_{i}.\qedhere
\end{align*}
\end{proof}

Now for any profile $\mathbf{p}\in\Q^{n!}$, there is an $n\times n$ matrix $Q_{\mathbf{p}}$ with all row and column sums equal to $N$ 
such that the possible ordinal outcomes are those whose associated faces intersect $Q_{\mathbf{p}}\overline{W}$. As the vectors in 
$\overline{W}$ are precisely the conical combinations of $\mathbf{v}_{1},\ldots,\mathbf{v}_{n-1}$, 
$Q_{\mathbf{p}}\overline{W}$ consists of the conical combinations of $\mathbf{s}_{1},\ldots,\mathbf{s}_{n-1}$ where 
$\mathbf{s}_{k}=Q_{\mathbf{p}}\mathbf{v}_{k}$. 
Writing $Q_{\mathbf{p}}=\big[\,\setlength\arraycolsep{3pt}\begin{matrix}\mathbf{q}_{1} & \cdots & \mathbf{q}_{n}\end{matrix}\,\big]$, 
we see that 
\[
\mathbf{s}_{k}=Q_{\mathbf{p}}\Big(\frac{k}{n}\one-\sum_{j=n-k+1}^{n}\mathbf{e}_{j}\Big) 
= \frac{k}{n}Q_{\mathbf{p}}\one - \sum_{j=n-k+1}^{n}Q_{\mathbf{p}}\mathbf{e}_{j} 
= \frac{kN}{n}\one - \sum_{j=n-k+1}^{n}\mathbf{q}_{j}.
\]
Since adding multiples of $\one$ will not change the face a vector lies in, it suffices to consider conical combinations of 
\[
\mathbf{t}_{k}=\mathbf{s}_{n-k}+\frac{Nk}{n}\one=N\one - \sum_{j=k+1}^{n}\mathbf{q}_{j} 
= \sum_{j=1}^{k}\mathbf{q}_{j},\quad k=1,\ldots,n-1.
\]
Also, scaling by a positive constant has no effect on which face a vector lies in, so the preceding observations can be stated as follows.

\begin{theorem}
\label{convex hull}
Let $\mathbf{p}\in\Q^{n!}$ be any profile and let $Q_{\mathbf{p}}=\sum_{\ell=1}^{n!}p_{\ell}R_{\ell}$ be given in column form by 
$Q_{\mathbf{p}}=\big[\,\setlength\arraycolsep{3pt}\begin{matrix}\mathbf{q}_{1} & \cdots & \mathbf{q}_{n}\end{matrix}\,\big]$. 
Define $\mathbf{t}_{k}=\sum_{j=1}^{k}\mathbf{q}_{j}$. Then the possible outcomes for a positional voting procedure with 
input $\mathbf{p}$ are those whose corresponding faces intersect the convex hull of $\mathbf{t}_{1},\ldots,\mathbf{t}_{n-1}$.
\end{theorem}

This suggests a way for a nefarious election official to obtain the most desirable possible outcome for themselves: 
Given the preferences of the electorate, construct the matrix $Q_{\mathbf{p}}$ and take $\mathbf{t}_{k}$ to be the sum of its first 
$k$ columns for $k=1,\ldots,n-1$. Then choose the most preferable outcome whose face intersects the convex hull of 
$\mathbf{t}_{1},\ldots,\mathbf{t}_{n-1}$, pick a point $\mathbf{r}$ in this intersection, and decompose it as 
$\mathbf{r}=\sum_{k=1}^{n-1}b_{k}\mathbf{t}_{k}$.
The favored outcome is assured by declaring the  weighting vector to be $\mathbf{w}=\sum_{k=1}^{n-1}b_{k}\mathbf{v}_{n-k}$. 

In practice, this could be accomplished by repeatedly generating a random probability vector 
$\mathbf{b} = [\,\setlength\arraycolsep{3pt}\begin{matrix} b_{1} & \cdots & b_{n-1}\end{matrix}\,]^\T$ and recording the 
ranking corresponding to the face containing $\mathbf{s}=T\mathbf{b}$, 
$T = \big[\,\setlength\arraycolsep{3pt}\begin{matrix}\mathbf{t}_{1} & \cdots & \mathbf{t}_{n-1}\end{matrix}\,\big]$. 
(This is just a matter of keeping track of the indices when $\mathbf{s}$ is sorted in descending order.)
After a sufficiently large number of iterations, one should have a nearly exhaustive list of possible ordinal rankings to choose from 
and can construct the desired weighting vector from the vector $\mathbf{b}$ corresponding to the favorite.

\begin{example} To illustrate this process, return to the profile given in Example~\ref{election}. We have
\[
Q_{\mathbf{p}}   = \begin{bmatrix}5 & 7 & 8 & 18 \\ 16 & 0 & 15 & 7 \\ 0 & 31 & 7 & 0 \\ 17 & 0 & 8 & 13
			        \end{bmatrix}.
			        \]
In Example~\ref{election}, we saw that using the Borda count, which corresponds to weighting vector $[1.5, 0.5, -0.5, -1.5]^\T$, 
the resulting societal ranking was $(3, 2, 4, 1)$, but using plurality, which corresponds to weighting vector $[0.75, -0.25, -0.25, -0.25]^\T$, 
the resulting societal ranking was $(4, 2, 1, 3)$. Suppose instead we want candidate 2 to win the election. We add the first $k$ columns of 
$Q_{\mathbf{p}}$ for $k=1, 2, 3$ to get
\[ 
\mathbf{t}_{1} = \begin{bmatrix} 5 \\ 16 \\ 0 \\ 17 \end{bmatrix}, 
\mathbf{t}_{2} = \begin{bmatrix} 12 \\ 16 \\ 31 \\ 17 \end{bmatrix}, 
\mathbf{t}_{3} = \begin{bmatrix} 20 \\ 31 \\ 38 \\ 25 \end{bmatrix}.
\]
Testing a few probability vectors reveals that  $b_1 = 0.6, b_2 = 0.2,$ and $b_3=0.2$ yields
\[ 
\mathbf{r} = b_{1} \mathbf{t}_{1} +  b_{2} \mathbf{t}_{2} + b_{3} \mathbf{t}_{3} = \begin{bmatrix} 9.4 \\ 19 \\ 13.8 \\ 18.6 \end{bmatrix}.
\]
Thus, we can achieve societal ranking $(2, 4, 3, 1)$ using the weighting vector
\begin{equation*}
\mathbf{w} = b_{1} \mathbf{v}_{3} + b_{2} \mathbf{v}_{2} + b_{3} \mathbf{v}_{1} 
= 0.6 \begin{bmatrix} 0.75 \\ -0.25 \\ -0.25 \\ -0.25 \end{bmatrix} +
0.2 \begin{bmatrix} 0.5 \\ 0.5 \\ -0.5 \\-0.5 \end{bmatrix} +
0.2 \begin{bmatrix} 0.25 \\ 0.25 \\ 0.25 \\ -0.75 \end{bmatrix}
 = \begin{bmatrix} 0.6 \\ 0 \\ -0.2 \\ -0.4 \end{bmatrix}.
\end{equation*}
\end{example}

\section*{Acknowledgment}
The authors wish to thank Marcelo Aguiar and Dan Katz for enlightening conversations. They are also 
grateful to the anonymous referees whose thoughtful comments improved the exposition substantially.

\bibliographystyle{plain}
\bibliography{voting}

\begin{thebibliography}{10}

\bibitem{AguMah}
Marcelo Aguiar and Swapneel Mahajan.
\newblock {\em Topics in Hyperplane Arrangements}, volume 226 of {\em
  Mathematical Surveys and Monographs}.
\newblock American Mathematical Society, Providence, RI, 2017.

\bibitem{BesRob}
J{\'e}r{\'e}my Besson and Celine Robardet.
\newblock A new way to aggregate preferences: Application to {E}urovision song
  contests.
\newblock In {\em Advances in Intelligent Data Analysis VII, 7th International
  Symposium on Intelligent Data Analysis}, volume 4723, pages 152--162, 2007.

\bibitem{Birk}
Garrett Birkhoff.
\newblock Three observations on linear algebra.
\newblock {\em Univ. Nac. Tucum\'{a}n. Revista A.}, 5:147--151, 1946.

\bibitem{Brua}
Richard~A. Brualdi.
\newblock Notes on the {B}irkhoff algorithm for doubly stochastic matrices.
\newblock {\em Canad. Math. Bull.}, 25(2):191--199, 1982.

\bibitem{CrisOrr}
Karl-Dieter Crisman and Michael~E. Orrison.
\newblock Representation theory of the symmetric group in voting theory and
  game theory.
\newblock In {\em Algebraic and Geometric Methods in Discrete Mathematics},
  volume 685 of {\em Contemp. Math.}, pages 97--115. Amer. Math. Soc.,
  Providence, RI, 2017.

\bibitem{DEMO}
Zajj Daugherty, Alexander~K. Eustis, Gregory Minton, and Michael~E. Orrison.
\newblock Voting, the symmetric group, and representation theory.
\newblock {\em Amer. Math. Monthly}, 116(8):667--687, 2009.

\bibitem{DufUcar}
Fanny Dufoss\'{e} and Bora U\c{c}ar.
\newblock Notes on {B}irkhoff--von {N}eumann decomposition of doubly stochastic
  matrices.
\newblock {\em Linear Algebra Appl.}, 497:108--115, 2016.

\bibitem{FraGro}
Jon Fraenkel and Bernard Grofman.
\newblock The {B}orda count and its real-world alternatives: Comparing scoring
  rules in {N}auru and {S}lovenia.
\newblock {\em Aust. J. Political Sci.}, 49, 2014.

\bibitem{HodKlim}
Jonathan~K. Hodge and Richard~E. Klima.
\newblock {\em The Mathematics of Voting and Elections: A Hands-On Approach},
  volume~30 of {\em Mathematical World}.
\newblock American Mathematical Society, Providence, RI, 2018.

\bibitem{Lang}
Serge Lang.
\newblock {\em Linear algebra}.
\newblock Undergraduate Texts in Mathematics. Springer-Verlag, New York, third
  edition, 1989.

\bibitem{Saari}
Donald~G. Saari.
\newblock Millions of election outcomes from a single profile.
\newblock {\em Soc. Choice Welf.}, 9(4):277--306, 1992.

\bibitem{Saari95}
Donald~G. Saari.
\newblock {\em Basic Geometry of Voting}.
\newblock Springer-Verlag, Berlin, 1995.

\bibitem{sage}
W.\thinspace{}A. Stein et~al.
\newblock {\em {S}age {M}athematics {S}oftware ({V}ersion 8.8)}.
\newblock The Sage Development Team, 2019.
\newblock {\tt www.sagemath.org}.

\bibitem{Terao}
Hiroaki Terao.
\newblock Chambers of arrangements of hyperplanes and {A}rrow's impossibility
  theorem.
\newblock {\em Adv. Math.}, 214(1):366--378, 2007.

\bibitem{vonN}
John von Neumann.
\newblock A certain zero-sum two-person game equivalent to the optimal
  assignment problem.
\newblock In {\em Contributions to the Theory of Games, Vol. 2}, Annals of
  Mathematics Studies, no. 28, pages 5--12. Princeton University Press,
  Princeton, N. J., 1953.

\end{thebibliography}

\end{document}